\documentclass[%12pt,
reqno]{amsart}%{article}
\usepackage{latexsym,amsmath,amssymb,amscd}
\usepackage[all]{xy}
\def\today{\ifcase \month \or
   January \or February \or March \or April \or
   May \or June \or July \or August \or
   September \or October \or November \or December \fi
   \space\number\day , \number\year}

\makeatletter
  \newcommand\@dotsep{4.5}
  \def\@tocline#1#2#3#4#5#6#7{\relax
     \ifnum #1>\c@tocdepth % then omit
     \else
     \par \addpenalty\@secpenalty\addvspace{#2}%
     \begingroup \hyphenpenalty\@M
     \@ifempty{#4}{%
     \@tempdima\csname r@tocindent\number#1\endcsname\relax
        }{%
         \@tempdima#4\relax
           }%
      \parindent\z@ \leftskip#3\relax \advance\leftskip\@tempdima\relax
      \rightskip\@pnumwidth plus1em \parfillskip-\@pnumwidth
       #5\leavevmode\hskip-\@tempdima #6\relax
       \leaders\hbox{$\m@th
       \mkern \@dotsep mu\hbox{.}\mkern \@dotsep mu$}\hfill
       \hbox to\@pnumwidth{\@tocpagenum{#7}}\par
       \nobreak
        \endgroup
         \fi}
\makeatother
\begin{document}

\makeatletter
\@addtoreset{figure}{section}
\def\thefigure{\thesection.\@arabic\c@figure}
\def\fps@figure{h,t}
\@addtoreset{table}{bsection}

\def\thetable{\thesection.\@arabic\c@table}
\def\fps@table{h, t}
\@addtoreset{equation}{section}
\def\theequation{%\thesection.
\arabic{equation}}
\makeatother

\newcommand{\bfi}{\bfseries\itshape}
\newtheorem{theorem}{Theorem}
\newtheorem{corollary}[theorem]{Corollary}
\newtheorem{criterion}[theorem]{Criterion}
\newtheorem{definition}[theorem]{Definition}
\newtheorem{example}[theorem]{Example}
\newtheorem{lemma}[theorem]{Lemma}
\newtheorem{notation}[theorem]{Notation}
\newtheorem{problem}[theorem]{Problem}
\newtheorem{proposition}[theorem]{Proposition}
\newtheorem{remark}[theorem]{Remark}
\numberwithin{theorem}{section}
\numberwithin{equation}{section}

%%% Todo
\newcommand{\todo}[1]{\vspace{5 mm}\par \noindent
\framebox{\begin{minipage}[c]{0.85 \textwidth}
\tt #1 \end{minipage}}\vspace{5 mm}\par}
%%%

\renewcommand{\1}{{\bf 1}}

\newcommand{\hotimes}{\widehat\otimes}

\newcommand{\Ad}{{\rm Ad}}
\newcommand{\Alt}{{\rm Alt}\,}
\newcommand{\Ci}{{\mathcal C}^\infty}
\newcommand{\comp}{\circ}
\newcommand{\wt}{\widetilde}

\newcommand{\ph}{\text{\bf P}}
\newcommand{\conv}{{\rm conv}}
\newcommand{\de}{{\rm d}}
\newcommand{\ev}{{\rm ev}}
\newcommand{\fimes}{\mathop{\times}\limits}
\newcommand{\id}{{\rm id}}
\newcommand{\ie}{{\rm i}}
\newcommand{\End}{{\rm End}\,}
\newcommand{\Gr}{{\rm Gr}}
\newcommand{\GL}{{\rm GL}}
\newcommand{\Hilb}{{\bf Hilb}\,}
\newcommand{\Hom}{{\rm Hom}}
\renewcommand{\Im}{{\rm Im}}
\newcommand{\Ker}{{\rm Ker}\,}
\newcommand{\Lie}{\textbf{L}}
\newcommand{\lf}{{\rm l}}
\newcommand{\Loc}{{\rm Loc}\,}
\newcommand{\pr}{{\rm pr}}
\newcommand{\Ran}{{\rm Ran}\,}
\newcommand{\supp}{{\rm supp}\,}

\newcommand{\Tr}{{\rm Tr}\,}
\newcommand{\Tran}{\textbf{Trans}}

\newcommand{\CC}{{\mathbb C}}
\newcommand{\RR}{{\mathbb R}}
\newcommand{\NN}{{\mathbb N}}

\newcommand{\G}{{\rm G}}
\newcommand{\U}{{\rm U}}
\newcommand{\Gl}{{\rm GL}}
\newcommand{\SL}{{\rm SL}}
\newcommand{\SU}{{\rm SU}}
\newcommand{\VB}{{\rm VB}}

\newcommand{\Ac}{{\mathcal A}}
\newcommand{\Bc}{{\mathcal B}}
\newcommand{\Cc}{{\mathcal C}}
\newcommand{\Dc}{{\mathcal D}}
\newcommand{\Ec}{{\mathcal E}}
\newcommand{\Fc}{{\mathcal F}}
\newcommand{\Gc}{{\mathcal G}}
\newcommand{\Hc}{{\mathcal H}}
\newcommand{\Kc}{{\mathcal K}}
\newcommand{\Nc}{{\mathcal N}}
\newcommand{\Oc}{{\mathcal O}}
\newcommand{\Pc}{{\mathcal P}}
\newcommand{\Qc}{{\mathcal Q}}
\newcommand{\Rc}{{\mathcal R}}
\newcommand{\Sc}{{\mathcal S}}
\newcommand{\Tc}{{\mathcal T}}
\newcommand{\Uc}{{\mathcal U}}
\newcommand{\Vc}{{\mathcal V}}
\newcommand{\Wc}{{\mathcal W}}
\newcommand{\Xc}{{\mathcal X}}
\newcommand{\Yc}{{\mathcal Y}}
\newcommand{\Zc}{{\mathcal Z}}
\newcommand{\Ag}{{\mathfrak A}}
\renewcommand{\gg}{{\mathfrak g}}
\newcommand{\hg}{{\mathfrak h}}
\newcommand{\mg}{{\mathfrak m}}
\newcommand{\nng}{{\mathfrak n}}
\newcommand{\pg}{{\mathfrak p}}
\newcommand{\Gg}{{\mathfrak g}}
\newcommand{\Lg}{{\mathfrak L}}
\newcommand{\Sg}{{\mathfrak S}}
\newcommand{\Ug}{{\mathfrak u}}

\markboth{}{}

\makeatletter
\title[Moment convexity of solvable locally compact groups]
{Moment convexity of solvable locally compact groups}
\author{Daniel Belti\c t\u a and Mihai Nicolae}
\address{Institute of Mathematics ``Simion
Stoilow'' of the Romanian Academy,
P.O. Box 1-764, Bucharest, Romania}
\email{Daniel.Beltita@imar.ro, beltita@gmail.com}
\email{mihaitaiulian85@yahoo.com}
%\thanks{This research was partially supported from the Grant of the Romanian National Authority for Scientific Research, CNCS-UEFISCDI, project number PN-II-ID-PCE-2011-3-0131. The first-named author also acknowledges partial support from the Project MTM2010-16679, DGI-FEDER, of the MCYT, Spain.}
\date{August 20, 2014}
%{\today. \textit{File name}: \texttt{BNic\_conv.tex}}
%\dedicatory{}

%\date{}
\makeatother

\begin{abstract}
We introduce moment maps for continuous unitary representations of general topological groups. 
For solvable separable locally compact groups, 
we prove that the closure of the image of the moment map of any representation is convex.  
\end{abstract}

\maketitle

%\tableofcontents

\section{Introduction}

For any continuous unitary representation 
$\pi\colon G\to U(\Hc)$ of a finite-dimensional Lie group 
its moment map (\cite{Wi89}, \cite{Mi90}, \cite{Wi92}) is 
$$\Psi_\pi\colon\Hc_\infty\setminus\{0\}\to\gg^*, \quad 
\Psi_\pi(v)=\frac{1}{\ie}\frac{\langle \de\pi(\cdot)v,v\rangle}{\langle v,v\rangle},$$
where $\Hc_\infty$ is the space of smooth vectors, and its moment set is 
$I_\pi:=\overline{\Psi_\pi(\Hc_\infty\setminus\{0\})}\subseteq\gg^*$, 
that is, the closure of the image of the moment map the dual space of the Lie algebra of~$G$. 
Representations of solvable Lie groups share a remarkable convexity property: 

\begin{theorem}[{\cite{AL92}}]\label{conv}
If $G$ is any solvable finite-dimensional Lie group, 
then for any continuous unitary representation $\pi\colon G\to U(\Hc)$  
its moment set $I_\pi$ is convex. 
\end{theorem}

Moment sets and moment maps were originally motivated by studies in symplectic geometry, 
and yet the above statement depends only on Lie theory and representation theory, 
so not on differential geometry. 
Therefore, in view of the successful Lie theory that was developed for various topological groups, 
including all connected locally compact groups
(see for instance \cite{HM07} and \cite{BCR81}), it is natural to ask the following question. 

\begin{problem}\label{probsolv}
\normalfont
Does Theorem~\ref{conv} carry over to general solvable topological groups, 
or at least to solvable pro-Lie groups? 
\end{problem}

The present paper offers an affirmative answer to the above problem 
in the case of separable locally compact groups (Theorem~\ref{main}). 
Our proof needs the locally compactness hypothesis via the Haar measure in Proposition~\ref{appr} and via the use of group $C^*$-algebras further on,  
and separability of these $C^*$-algebras is required for certain arguments involving direct integrals of representations 
in Lemma~\ref{disintegr} and the proof of Proposition~\ref{other}. 
Despite these technical restrictions, we found it useful to introduce the notion of a moment map for continuous unitary representations of arbitrary topological groups (Definition~\ref{momdef}) and to derive some of its basic properties in that general setting. 
Note that Theorem~\ref{conv} trivially holds true if $\pi$ is replaced by any representation 
of a Banach-Lie group for which $\Hc_\infty=\{0\}$ (see \cite{BNe08} for a specific example). 
We will explore elsewhere other instances when Problem~\ref{probsolv} can be answered in the affirmative and 
its relations to the method of coadjoint orbits in the spirit of \cite{Wi89}. 

\section{Preliminaries} 

\subsection*{Lie theory for topological groups} 
We use differential calculus on topological groups as developed in \cite{BCR81}; 
see also \cite{BB11} and \cite{BNic14}. 

Unless otherwise mentioned, $G$ is any topological group  
and $\Hc$ is any complex Hilbert space with its unitary group $U(\Hc):=\{T\in\Bc(\Hc)\mid TT^*=TT^*=\1\}$ 
regarded as a topological group with respect to the strong operator topology 
(i.e., the topology of pointwise convergence).  
Denoting by $\Cc(\cdot,\cdot)$ the spaces of continuous maps, 
one defines 
$$\Lg(G):=\{\gamma\in\Cc(\RR,G)\mid 
(\forall t,s\in{\mathbb R})\quad \gamma(t+s)=\gamma(t)\gamma(s)\}.$$  
This set is endowed with the topology of uniform convergence on compact subsets of~$\RR$. 
If $\varphi\colon H\to G$ is any continuous homomorphism of topological groups, 
then one defines 
$$\Lg(\varphi)\colon\Lg(H)\to\Lg(G),\quad \gamma\mapsto\varphi\circ\gamma$$
and it is easily checked that one thus obtains a functor $\Lg(\cdot)$ 
from the category of topological groups to the category of topological spaces.  

The \emph{adjoint action} of the topological group $G$ is the mapping 
\begin{equation}\label{adcont}
\Ad_G\colon G\times\Lg(G)\to\Lg(G),\quad (g,\gamma)\mapsto\Ad_G(g)\gamma:=g\gamma(\cdot)g^{-1},
\end{equation}
which is continuous \cite[Prop. 2.28]{HM07} and homogeneous in the sense that 
$\Ad_G(g)(r\cdot \gamma)=r\cdot(\Ad_G(g)\gamma)$ for all $r\in\RR$, $g\in G$, and $\gamma\in\Lg(G)$, 
where one defines 
$$(\forall r,t\in\RR)(\forall \gamma\in\Lg(G))\quad 
(r\cdot \gamma)(t):=\gamma(rt). $$
For $r=-1$ and $\gamma\in\Lg(G)$ we denote $-\gamma:=(-1)\cdot\gamma\in\Lg(G)$. 

Let an arbitrary open subset $V\subseteq G$ 
and $\Yc$ be any real locally convex space.  
%\begin{enumerate}
%\item\label{aux4_item1}  
If $\varphi\colon V\to\Yc$, $\gamma\in\Lg(G)$, and $g\in V$, then we denote 
\begin{equation}\label{aux4_eq1}
(D_\gamma\varphi)(g):=\lim_{t\to 0}\frac{\varphi(g\gamma(t))-\varphi(g)}{t} 
\end{equation}
if the limit in the right-hand side exists. 
Similarly for 
\begin{equation}\label{aux4_eq1_R}
(D^R_\gamma\varphi)(g):=\lim_{t\to 0}\frac{\varphi(\gamma(-t)g)-\varphi(g)}{t} 
=(D_{-\Ad_G(g^{-1})\gamma}\varphi)(g). 
\end{equation}
%\item\label{aux4_item2}  
One defines $\Cc^1(V,\Yc)$ as the set of all functions $\varphi\in\Cc(V,\Yc)$ 
for which the function 
$$
D\varphi\colon V\times\Lg(G)\to\Yc,\quad 
(D\varphi)(g;\gamma):=(D_\gamma\varphi)(g) $$
is well defined and continuous. 
One also denotes $D\varphi=:D^1\varphi$. 

Now let $n\ge 2$ and assume the space $\Cc^{n-1}(V,\Yc)$ and the mapping 
$D^{n-1}$ have been defined. 
Then $\Cc^n(V,\Yc)$ is defined as the set of all $\varphi\in\Cc^{n-1}(V,\Yc)$ 
for which the function 
$$\begin{aligned}
D^{n}\varphi\colon 
V\times \Lg(G)\times\cdots\times\Lg(G) & \to\Yc,\\
(g;\gamma_1,\dots,\gamma_n) & \mapsto (D_{\gamma_n}(D_{\gamma_{n-1}}\cdots(D_{\gamma_1}\varphi)\cdots))(g)
\end{aligned} $$
is well defined and continuous. 

Moreover $\Ci(V,\Yc):=\bigcap\limits_{n\ge1}\Cc^n(V,\Yc)$ 
and $\Cc^\infty_0(V,\Yc)$ is the set of all $\varphi\in\Ci(V,\Yc)$ 
having compact support. 
%\end{enumerate}
If $\Yc=\CC$, then we write simply $\Cc^n(G):=\Cc^n(V,\CC)$    
etc.,  
for $n=1,2,\dots,\infty$. 

It will be convenient to use the notations 
$$\begin{aligned}
D_\gamma\varphi:= 
&D_{\gamma_n}(D_{\gamma_{n-1}}\cdots(D_{\gamma_1}\varphi)\cdots)\colon G\to\Yc \\
D^R_\gamma\varphi:= 
&D^R_{\gamma_n}(D^R_{\gamma_{n-1}}\cdots(D^R_{\gamma_1}\varphi)\cdots)\colon G\to\Yc
\end{aligned}$$
whenever $\gamma:=(\gamma_1,\dots,\gamma_n)\in \Lg(G)\times\cdots\times\Lg(G)$ and $\varphi\in\Cc^n(G,\Yc)$. 

\begin{definition}[{\cite[Def. 3.1]{BNic14}}]\label{topology}
\normalfont
We endow the function space $\Ci(V,\Yc)$ with the locally convex topology defined by the family of 
seminorms $p_{K_1,K_2}$ introduced as follows. 
Denote 
$$(\forall k\ge 1)\quad \Lg^k(G):=\underbrace{\Lg(G)\times\cdots\times\Lg(G)}_{k\text{ times}}.$$
For every $k\ge 1$, any compact subsets $K_1\subseteq \Lg^k(G)$ and $K_2\subseteq V$, 
and any continuous seminorm $\vert\cdot\vert$ on~$\Yc$, 
 we define 
$p^{\vert\cdot\vert}_{K_1,K_2}\colon\Ci(V,\Yc)\to[0,\infty)$ 
by  
$$p^{\vert\cdot\vert}_{K_1,K_2}(f):=
\begin{cases}
\sup\{\vert (D_\gamma f)(x)\vert\mid \gamma\in K_1,x\in K_2\} &\text{ if }K_1\ne\emptyset,\\
\sup\{\vert f(x)\vert\mid x\in K_2\} &\text{ if }K_1=\emptyset.
\end{cases}$$
For the sake of simplicity we always omit the seminorm $\vert\cdot\vert$ on~$\Yc$ 
from the above notation, by writing simply $p_{K_1,K_2}$ instead of $p^{\vert\cdot\vert}_{K_1,K_2}$.
\end{definition}

\begin{proposition}\label{ch1P2.4}
If $G$ is any topological group and $\Yc$ is any locally convex space,  
then the following assertions hold. 
\begin{enumerate}
\item\label{ch1P2.4_item1} Fix any open set $V\subseteq G$. 
If for all $k\ge 1$, compact sets $K_1\subseteq \Lg^k(G)$ and $K_2\subseteq V$, 
and continuous seminorm $\vert\cdot\vert$ on~$\Yc$ one 
defines 
$q^{\vert\cdot\vert}_{K_1,K_2}\colon\Ci(V,\Yc)\to[0,\infty)$ by 
$$q^{\vert\cdot\vert}_{K_1,K_2}(f):=
\begin{cases}
\sup\{\vert (D^R_\gamma f)(x)\vert\mid \gamma\in K_1,x\in K_2\} &\text{ if }K_1\ne\emptyset,\\
\sup\{\vert f(x)\vert\mid x\in K_2\} &\text{ if }K_1\ne\emptyset, 
\end{cases}$$
then one thus obtains a family of seminorms that determines the topology of $\Ci(V,\Yc)$ 
introduced in Definition~\ref{topology}. 
\item\label{ch1P2.4_item2} For all $g\in G$ and $\gamma\in\Lg(G)$, the operators 
$$D_\gamma,D^R_\gamma,\lambda(g),\rho(g)\colon \Ci(G,\Yc)\to \Ci(G,\Yc)$$ 
are well-defined and continuous, 
where 
$$(\lambda(g)f)(x):=f(gx)\text{ and }(\rho(g)f)(x):=f(xg)\text{ for all }x\in G\text{ and }f\in\Ci(G,\Yc).$$ 
\end{enumerate}
\end{proposition}

\begin{proof}
For the proof it is convenient to denote for all $k\ge 1$, $\gamma=(\gamma_1,\dots,\gamma_k)\in\Lg^k(G)$, 
and $x\in G$, 
$$\begin{aligned}
\gamma^x:=
&(\Ad_G(x^{-1})\gamma_1,\dots,\Ad_G(x^{-1})\gamma_k),  \\
-\gamma:=
&(-\gamma_1,\dots,-\gamma_k). 
\end{aligned}$$
We now proceed with the proof of the two assertions from the statement. 

\eqref{ch1P2.4_item1}
Let $K_2\subseteq G$ and $K_1\subseteq \Lambda^k(G)$ be any compact sets and $\vert\cdot\vert$ be any continuous seminorm on $\Yc$, which will be however omitted from the notation, writing simply 
$q^{\vert\cdot\vert}_{K_1,K_2}=:q_{K_1,K_2}$. 

It follows by \eqref{aux4_eq1_R} that for all $f\in\Ci(V,\Yc)$ and $x\in V$ one has 
\begin{equation}\label{ch1P2.4_proof_eq1}
(D_{\gamma} f)(x)=(D^R_{-\gamma^x} f)(x)
\end{equation} 
hence 
$p_{K_1,K_2}(f) \le q_{K_3,K_2}(f) $ where  
$$K_3:=\{ -\gamma^x \mid \gamma\in K_1,x\in K_2 \} .$$
Also, writing \eqref{ch1P2.4_proof_eq1} as  
$ (D_{\gamma}^R f)(x)=(D_{-\gamma^{x^{-1}}} f)(x) $, one obtains
$q_{K_1,K_2}(f) \le p_{K_4,K_2}(f) $, where  
$$K_4:=\{ -\gamma^{x^{-1}} \mid \gamma\in K_1,x\in K_2 \} .$$
To conclude the proof, it remains to show that $K_3$ and $K_4$ are compact subsets of $\Lg^k(G)$. 
The set $K_3$ is the image of the compact set $K_2\times K_1$ through the map 
$$G\times \Lambda^k(G)\to \Lambda^k(G) , \quad (x,\gamma)\mapsto -\gamma^x,$$
while the set $K_4$ is the image of the compact set $K_2 \times K_1 $ 
through the map 
$$G\times \Lambda^k(G)\to \Lambda^k(G) , \quad (x,\gamma)\mapsto -\gamma^{x^{-1}}. $$
Both these maps are continuous since \eqref{adcont} is continuous. 

\eqref{ch1P2.4_item2}
For the operator $\lambda(g)\colon \Ci(G,\Yc)\to\Ci(G,\Yc)$ we use the method of proof of \cite[Prop. 3.9]{BNic14}, 
which establishes that $\lambda(g)$ is well defined.  
Moreover, for any $k\ge 1$ one has the homeomorphism
$$F_k^g\colon \Lg(G)\times\cdots\times\Lg(G)\times G\to\Lg(G)\times\cdots\times\Lg(G)\times G,\quad  
(\beta,x)\mapsto(\beta,gx). $$
Iterating \cite[Rem. 3.8]{BNic14}, one obtains $D^k(\lambda(g)f)=(D^k f)\circ F_k^g$ 
for all $f\in\Ci(G,\Yc)$, and this shows as above that for  any continuous seminorm $\vert\cdot\vert$ on $\Yc$ 
and all compact sets 
$K_2\subseteq G$ and $K_1\subseteq \Lambda^k(G)$ one has 
$p_{K_1,K_2}(\lambda(g)f)=p_{K_1,K_2'}(f)$, 
where $K_2':=gK_2$
hence the linear operator $\lambda(g)\colon \Ci(G,\Yc)\to\Ci(G,\Yc)$ is continuous. 

For the operator $\rho(g)$ one can similarly obtain 
$(D^R)^k(\rho(g)f)=((D^R)^k f)\circ \widetilde{F}_k^g$, where 
$$\widetilde{F}_k^g\colon \Lg(G)\times\cdots\times\Lg(G)\times G\to\Lg(G)\times\cdots\times\Lg(G)\times G,\quad
(\beta,x)\mapsto(\beta,xg). $$
It follows that 
$q_{K_1,K_2}(\rho(g)f)=q_{K_1,K_2''}(f)$, where $K_2'':=K_2g$, 
 for  any continuous seminorm $\vert\cdot\vert$ on $\Yc$,  
all compact sets 
$K_2\subseteq G$ and $K_1\subseteq \Lambda^k(G)$, 
and any $f\in\Ci(G,\Yc)$. 
Since the seminorms $q_{K_1,K_2}$ determine the topology of $\Ci(G,\Yc)$ by Assertion~\eqref{ch1P2.4_item1}, 
we see that the linear operator $\rho(g)\colon \Ci(G,\Yc)\to\Ci(G,\Yc)$ is continuous. 

If $\gamma\in\Lg(G)$ then for every $f\in\Ci(G,\Yc)$ one has 
$p_{K_1,K_2}(D_\gamma f)=p_{K_1',K_2}(f)$ and $q_{K_1,K_2}(D^R_\gamma f)=q_{K_1',K_2}(f)$, where $K_1':=\{\gamma\}\times K_1\subseteq \Lg^{k+1}(G)$, 
with the convention $\{\gamma\}\times\emptyset=\{\gamma\}$. 
This shows that both the linear operators $D_\gamma,D^R_\gamma\colon \Ci(G,\Yc)\to\Ci(G,\Yc)$ are continuous, 
and the proof ends. 
\end{proof}

In the following we use \emph{pro-Lie groups} (that is, topological groups that are isomorphic to limits of some projective systems of Lie groups) and their Lie theory, for which we refer to the excellent monograph \cite{HM07}. 
Pro-Lie groups were called LP-groups in \cite{BCR81}, \cite{Bos76}. 

\begin{lemma}[Yamabe's theorem]\label{lc-pro}
Connected locally compact groups are pro-Lie groups. 
\end{lemma}

\begin{proof}
This follows by \cite[Th. 3.39 (Def. A $\Rightarrow$ Def. B)]{HM07}, since 
every locally compact group is complete \cite[Rem. 1.31]{HM07} and 
every connected locally compact group contains arbitrarily small co-Lie subgroups 
by \cite[Th. 4.6, pag. 175]{MZ55}. 
\end{proof}

\begin{lemma}\label{pro-pre}
Every pro-Lie group $G$ is a pre-Lie group, in the sense that: 
\begin{enumerate}
\item The topological space $\Lg(G)$ has the structure of a locally convex Lie algebra over~$\RR$, 
whose scalar multiplication, vector addition and bracket 
satisfy the following conditions for all 
$t,s\in{\mathbb R}$ and $\gamma_1,\gamma_2\in\Lg(G)$: 
\begin{eqnarray}
\nonumber%\label{scalar} 
 (t\cdot \gamma_1)(s) &=& \gamma_1(ts);  \\
\nonumber%\label{addition}
(\gamma_1+\gamma_2)(t) &=& \lim\limits_{n\to\infty}(\gamma_1(t/n)\gamma_2(t/n))^n;\\
\nonumber%\label{bracket}
[\gamma_1,\gamma_2](t^2) &=& \lim\limits_{n\to\infty}(\gamma_1(t/n)\gamma_2(t/n)\gamma_1(-t/n)\gamma_2(-t/n))^{n^2}, 
\end{eqnarray}
where the convergence is  uniform on compact subsets of $\RR$. 
\item 
For every nontrivial $\gamma\in\Lg(G)$ there exists a function $\varphi$ of class $\Ci$ on 
some neighborhood of $\1\in G$ with $(D_\gamma\varphi)(\1)\ne0$.  
\end{enumerate} 
\end{lemma}

\begin{proof}
See \cite[Prop. 1.3.1(3)]{BCR81}.
\end{proof}

\subsection*{Smooth vectors for representations of topological groups}
By \emph{representation} we mean a continuous unitary representation,  
hence a group homomorphism $\pi\colon G\to U(\Hc)$
for which the map $G\times\Hc\to\Hc$, $(g,y)\mapsto\pi(g)y $ 
is continuous. 

The space of smooth vectors has been widely used in representation theory of finite-dimensional Lie groups; 
see also \cite{Ne10} for a deep study of smooth vectors of representations of infinite-dimensional Lie groups. 
This space was introduced in representation theory of general topological groups in  \cite{Bos76} 
(see also \cite{BB11} for more versions of this notion). 

\begin{definition}\label{smvect}
\normalfont 
The space of \emph{smooth vectors} for the representation $\pi$ is
$$
\Hc_\infty:=
\{v\in\Hc\mid \pi(\cdot)v\in\Ci(G,\Hc)\}.
$$
One endows $\Hc_\infty$ with the locally convex topology for which the injective linear map 
\begin{equation}\label{smvect_eq1}
A\colon \Hc_\infty\to \Ci(G,\Hc),\quad v\mapsto\pi(\cdot)v
\end{equation}
is a homeomorphism onto its image, where the image of the above map is regarded as a subspace of the locally convex space $\Ci(G,\Hc)$ endowed with the topology introduced in Definition~\ref{topology}. 
\end{definition}

Before continuing we note that if $G$ is a Banach-Lie group, then the above topology on the space of smooth vectors is in general different from the topology introduced in \cite[Sect. 4]{Ne10}, 
however these topologies agree in the case of finite-dimensional Lie groups. 

\begin{remark}\label{top}
\normalfont
We record for later use the following easy observations. 
\begin{enumerate}
\item\label{top_item1} $\Ran A=\{\psi\in\Ci(G,\Hc)\mid (\forall x,y\in G)\ \psi(xy)=\pi(x)\psi(y)\}$ 
and this is a closed linear subspace of $\Ci(G,\Hc)$. 
\item\label{top_item2} For every $\psi\in\Ran A$ and $y\in G$ one has $\psi(y)\in\Hc_\infty$, and $\psi(\1)=A^{-1}(\psi)$. 
\item\label{top_item3} The set $\Hc_\infty$ is a 
linear subspace of $\Hc$ that is invariant under $\pi(x)$ and $\de\pi(\gamma)$ 
for all $x\in G$ and $\gamma\in\Lg(G)$, and the operators $\pi(x),\de\pi(\gamma)\colon \Hc_\infty\to\Hc_\infty$ 
are continuous (using Proposition~\ref{ch1P2.4}). 
\item\label{top_item4} The inclusion map $\Hc_\infty\hookrightarrow\Hc$ is continuous since it is equal to the composition of $A$ with the evaluation map 
$\Ci(G,\Hc)\to\Hc$, $\psi\mapsto\psi(\1)$, which is continuous. 
\item\label{top_item5} The operator $A$ intertwines the representation  $\pi(\cdot)\vert_{\Hc_\infty}$ 
and the right regular representation of $G$ on $\Ci(G,\Hc)$, 
that is, $(A(\pi(x)v))(y)=(Av)(yx)$ for all $x,y\in G$ and $v\in\Hc$. 
For $\gamma\in\Lg(G)$ and $x=\gamma(t)$ with $t\in\RR$, 
this implies $(A(\de\pi(\gamma)v))(y)=(D_\gamma(Av))(y)$   
(see \eqref{aux4_eq1}), that is, $A(\de\pi(\gamma)v)=D_\gamma(Av)$. 
\end{enumerate}
\end{remark}

\begin{remark}\label{intert}
\normalfont
Let $\pi_j\colon G\to U(\Hc_j)$ be any representation with its corresponding space of smooth vectors 
$\Hc_{j,\infty}$ for $j=1,2$. 
If $T\colon \Hc_1\to\Hc_2$ is a bounded linear operator with the property 
$T\pi_1(x)=\pi_2(x)T$ for all $x\in G$, then $T(\Hc_{1,\infty})\subseteq \Hc_{2,\infty}$. 
\end{remark}

\begin{definition}\label{deriv_repr}
\normalfont
{\it The derivative of the representation} $\pi\colon G\to U(\Hc)$ is  
$$\de\pi\colon\Lg(G)\to\End(\Hc_\infty),\quad 
\de\pi(\gamma)y:=\lim_{t\to 0}\frac{\pi(\gamma(t))y-y}{t}  $$
for every $\gamma\in\Lg(G)$ and $y\in\Hc_\infty$. 
Moreover, for every $k\ge 1$ and $\gamma=(\gamma_1,\dots,\gamma_k)\in\Lg^k(G)$ we define the linear operator
$$\de\pi(\gamma)\colon\Hc_\infty\to\Hc_\infty,\quad 
\de\pi(\gamma)y:=\de\pi(\gamma_1)\cdots\de\pi(\gamma_k)y. $$
\end{definition}

\begin{lemma}\label{unifcomp}
Let $G$ be any topological group and $\pi\colon G\to U(\Hc)$ be any unitary representation. 
Then for every integer $k\ge 1$, compact set $K\subseteq\Lg^k(G)$, and vector $v\in\Hc_\infty$, 
one has 
\begin{equation}\label{unifcomp_eq1}
\lim\limits_{x\to\1}\sup\limits_{\gamma\in K}\Vert\de\pi(\gamma)v-\de\pi(\Ad_G(x)\gamma)v\Vert=0,
\end{equation} 
where $\Ad_G\colon G\times \Lg^k(G)\to \Lg^k(G)$ is defined componentwise. 
\end{lemma}

\begin{proof}
For $\gamma=(\gamma_1,\dots,\gamma_k)\in\Lg^k(G)$ one has 
$$\begin{aligned}
\de\pi(\gamma)v-\de\pi(\Ad_G(x)\gamma)v
=
&\de\pi(\gamma_1)\cdots\de\pi(\gamma_k)v-\de\pi(\Ad_G(x)\gamma_1)\cdots\de\pi(\Ad_G(x)\gamma_k)v 
\end{aligned}$$
hence by Remark~\ref{top}\eqref{top_item5}, 
$$A(\de\pi(\gamma)v-\de\pi(\Ad_G(x)\gamma)v)=D^k_\gamma(Av)-D^k_{\Ad_G(x)\gamma}(Av).$$
Therefore \eqref{unifcomp_eq1} is equivalent to the fact that for $f:=Av\in\Ci(G,\Hc)$ one has 
$$\lim\limits_{x\to\1}\sup\limits_{\gamma\in K}\Vert (D^k_\gamma f-D^k_{\Ad_G(x)\gamma} f)(\1)\Vert=0,$$
that is, 
$$\lim\limits_{x\to\1}\sup\limits_{\gamma\in K}\Vert (D^kf)(\1,\gamma)-D^kf(\1,\Ad_G(x)\gamma)\Vert=0. $$
To prove the above equality, 
define 
$$\varphi\colon G\times\Lg^k(G)\to\Lg^k(G),\quad 
\varphi(x,\gamma):=D^kf(\1,\Ad_G(x)\gamma).$$
Since $f\in\Ci(G,\Hc)$, the function $D^kf(\1,\cdot)\colon\Lg^k(G)\to\Hc$ is continuous, 
hence composing it with the continuous map $\Ad_G\colon G\times\Lg^k(G)\to\Lg^k(G)$, 
one obtains that $\varphi$ is continuous. 
Then the map $G\to\Cc(K,\Hc)$, $x\mapsto \varphi(x,\cdot)\vert_K$ is continuous by \cite[Th. 1]{Fo45}, 
and we are done. 
\end{proof}

We will need the following result which might be regarded as a very simple version of the Dixmier-Malliavin theorem \cite{DM78}, 
saying that the G\aa rding space is equal to the space of smooth vectors for all unitary representations of finite-dimensional Lie groups. 
Throughout the following, by \emph{smooth $\delta$-family} on a locally compact group $G$ we mean 
a family of functions $\{f_W\}_{W\in\Wc}$, where $\Wc$ is a neighboorhood base at $\1\in G$, 
and for every $W\in\Wc$ one has  $\supp f_W\subseteq W$, $0\le f_W\in\Ci(G)$, 
and $\int_G f_W(x)\de x=1$.

\begin{proposition}\label{appr}
Let $G$ be any locally compact group $G$ and $\pi\colon G\to U(\Hc)$ be any representation
extended to 
$$\pi\colon L^1(G)\to\Bc(G),\quad \pi(f)=\int_G f(x)\pi(x)\de x. $$ 
If $f\in\Cc^\infty_0(G)$ and $v\in\Hc$, then $\pi(f)v\in\Hc_\infty$ and 
\begin{equation}\label{appr_eq1}
(\forall\gamma\in\Lg(G))\quad \de\pi(\gamma)(\pi(f)v)=\pi(D^R_\gamma f)v. 
\end{equation}
Moreover, for every smooth $\delta$-family $\{f_W\}_{W\in\Wc}$  
the set 
$\{\pi(f_W)v\mid  W\in\Wc,\ v\in\Hc\} $
is a dense subset of $\Hc_\infty$. 
\end{proposition}

\begin{proof}
If $\gamma\in\Lg(G)$, $t\in\RR$, $v\in\Hc$, and $f\in\Cc^\infty_0(G)$, then 
$\pi(f)v\in\Hc_\infty$ by \cite[Lemma 3.1]{Bos76}, and moreover 
$$\pi(\gamma(t))(\pi(f)v)= \int_G f(x)\pi(\gamma(t)x)v\de x=\int_G f(\gamma(-t)x)\pi(x)v\de x$$
hence by differentiation at $t=0$ one obtains \eqref{appr_eq1}. 
For all $v\in\Hc$ and $W\in\Wc$ define $v_W:=\pi(f_W)v \in\Hc_\infty$. 
One has 
\begin{equation}\label{appr_proof_eq1}
v-v_W=\int_G f_W(x)v\de x-\int_G f_W(x)\pi(x)v\de x=\int_W f_W(x)(v-\pi(x)v)\de x 
\end{equation}
hence using $0\le f_W\in\Ci(G)$  
and $\int_G f_W(x)\de x=1$, one obtains 
$$\Vert v-v_W\Vert\le \int_W f_W(x)\Vert v-\pi(x)v\Vert\de x\le \sup_{x\in W}\Vert v-\pi(x)v\Vert $$
and therefore $\lim\limits_W\Vert v-v_W \Vert=0$. 
On the other hand, using \eqref{smvect_eq1}, 
$$(\forall y\in G)\quad (A(v-v_W))(y)=\pi(y)(v-v_W)
$$
and this implies 
$\lim\limits_W\sup\limits_{y\in G}\Vert(A(v-v_W))(y)\Vert
=\lim\limits_W\Vert v-v_W \Vert=0$. 

If $v\in\Hc_\infty$, $k\ge 1$, and $\gamma\in\Lg^k(G)$, 
one has $D_\gamma A=A\circ\de\pi(\gamma)\colon\Hc_\infty\to\Ci(G,\Hc)$ 
by Remark~\ref{top}\eqref{top_item5}, 
hence using \eqref{appr_proof_eq1}, one obtains 
$$\begin{aligned}
(D^k_\gamma A(v-v_W))(y)
&=(A(\de\pi(\gamma)(v-v_W)))(y) \\
&=\int_W f_W(x)\pi(y)(\de\pi(\gamma)v-\de\pi(\gamma)\pi(x)v)\de x \\
&=\int_W f_W(x)\pi(y)(\de\pi(\gamma)v-\pi(x)\de\pi(\Ad_G(x^{-1})\gamma)v)\de x.
\end{aligned}$$
Note that under the above integral one can write 
$$\begin{aligned}
\de\pi(\gamma)v-\pi(x)\de\pi(\Ad_G(x^{-1})\gamma)v
=
&(\1-\pi(x))\de\pi(\gamma)v \\
&+\pi(x)(\de\pi(\gamma)v-\de\pi(\Ad_G(x^{-1})\gamma)v),  
\end{aligned}$$ 
hence using again $0\le f_W\in\Ci(G)$, 
and $\int_G f_W(x)\de x=1$, one obtains for all $y\in G$, 
$$\Vert (D^k_\gamma A(v-v_W))(y)\Vert
\le  \sup_{x\in W}(\Vert (\1-\pi(x))\de\pi(\gamma)v\Vert +\Vert\de\pi(\gamma)v-\de\pi(\Ad_G(x^{-1})\gamma)v\Vert).$$
It then easily follows by Lemma~\ref{unifcomp} that $\lim\limits_W A(v-v_W)=0$ in $\Ci(G,\Hc)$ 
(see Definition~\ref{topology}), that is, 
$\lim\limits_W v_W=v$ in $\Hc_\infty$. 
\end{proof}

\begin{corollary}\label{pro}
Let $G$ be any connected locally compact group and 
$\Nc_0(G)$ be the family of compact normal subgroups $N\subseteq G$ for which $G/N$ is a Lie group. 
Fix any $N_0\in\Nc_0(G)$ with the corresponding projection $p\colon G\to G/N_0$, 
and let $\pi_0\colon G/N_0\to U(\Hc)$ be any representation 
with its space of smooth vectors~$\Hc_\infty(\pi_0)$. 

Then for every $N\in\Nc_0(G)$ with $N\subseteq N_0$ the space of smooth vectors of the representation 
$G/N\to U(\Hc)$, $gN\mapsto\pi_0(gN_0)$ is equal to $\Hc_\infty(\pi_0)$. 
Also, $\Hc_\infty(\pi_0)$ is a dense subspace of the space $\Hc_\infty(\pi)$ 
of smooth vectors of the representation 
$\pi:=\pi_0\circ p\colon G\to U(\Hc)$. 
\end{corollary}

\begin{proof}
The first assertion is clear since the homomorphism 
$G/N\to G/N_0$, $xN\mapsto xN_0$ is a covering map.  
For the second assertion, 
recalling the canonical isomorphism of $G$ onto the projective limit of Lie groups 
$\lim\limits_{N_0\subseteq N}G/N$ (see \cite{HM07}),  
one can use Proposition~\ref{appr} for a suitable $\delta$-family consisting of functions of the form 
$x\mapsto f(xN)$, where $N\in\Nc_0(G)$ with $N\subseteq N_0$, and $f\in\Cc_0^\infty(G/N)$. 
\end{proof}

\section{Moment sets of representations of topological groups}

\begin{definition}\label{momdef}
\normalfont
For any representation of a topological group $\pi\colon G\to U(\Hc)$, 
its {\it moment map} 
is 
$$\Psi_\pi\colon \Hc_\infty\setminus\{0\}\to\RR^{\Lg(G)},\quad 
\Psi_\pi(v)=\frac{1}{\ie}\frac{\langle\de\pi(\cdot)v, v\rangle}{\langle v, v\rangle}.$$
The set $\RR^{\Lg(G)}$ of all functions $\Lg(G)\to\RR$ 
is endowed with the topology of pointwise convergence, 
and one defines the \emph{closed moment set~$I_\pi$} of the representation~$\pi$ as the closure 
in $\RR^{\Lg(G)}$  of the image of the moment map $I^0_\pi:=\Psi_\pi(\Hc_\infty\setminus\{0\})$, 
that is, $I_\pi:=\overline{I^0_\pi}$. 
\end{definition}

\begin{lemma}\label{lincont}
Let $G$ be any pro-Lie group. 
Then for every representation $\pi\colon G\to U(\Hc)$ and every $v\in\Hc_\infty\setminus\{0\}$ the function 
$\Psi_\pi(v)\colon\Lg(G)\to\RR$ is linear and continuous. 
\end{lemma}

\begin{proof}
It was noted in \cite[Anm. 2.1, page 235]{Bos76} that $\de\pi(\cdot)v\colon\Lg(G)\to\Hc$ is a continuous linear map, 
hence our assertion follows at once. 
\end{proof}

For every real topological vector space $\Yc$ we denote by $\Yc^*$ the weak dual of $\Yc$, 
that is, the space of all continuous linear functionals on $\Yc$ endowed with the topology of pointwise convergence. 
Then Lemma~\ref{lincont} shows that for any representation $\pi\colon G\to U(\Hc)$ of a pro-Lie group 
one has $\Psi_\pi\colon\Hc_\infty\setminus\{0\}\to\Lg(G)^*$. 

\begin{lemma}\label{cont}
If $G$ is any 
topological group $G$, 
then for every representation $\pi\colon G\to U(\Hc)$ 
its moment map $\Psi_\pi\colon\Hc_\infty\setminus\{0\}\to\RR^{\Lg(G)}$ is continuous. 
\end{lemma}

\begin{proof}
We must check that for every $\gamma\in\Lg(G)$ the function 
$$\Hc_\infty\setminus\{0\}\to\RR, \quad 
v\mapsto(\Psi_\pi(v))(\gamma)=\frac{1}{\ie}\frac{\langle\de\pi(\gamma)v, v\rangle}{\langle v, v\rangle}$$ 
is continuous. 
This follows since both the map $\de\pi(\gamma)\colon \Hc_\infty\to\Hc_\infty$ 
and the inclusion map $\Hc_\infty\hookrightarrow\Hc$ 
are continuous (Remark~\ref{top}).  
Compare also \cite[Lemmas 6.1--6.3]{Bos76}. 
\end{proof}

\begin{proposition}\label{directed}
Let $\pi\colon G\to U(\Hc)$ be any representation and $\{\Hc_j\}_{j\in J}$ be any 
family of closed linear subspaces of $\Hc$ satisfying the following conditions: 
\begin{enumerate}
\item  For every $j_1,j_2\in J$ there exists $j_3\in J$ with $\Hc_{j_1}+\Hc_{j_2}\subseteq\Hc_{j_3}$, 
and one defines $j_1\le j_2$ if and only if $\Hc_{j_1}\subseteq\Hc_{j_2}$. 
\item  If $p_j\colon \Hc\to\Hc_j$ is the orthogonal projection onto $\Hc_j$ for arbitrary $j\in J$, 
then 
\begin{equation}\label{directed_eq1}
(\forall v\in\Hc)\quad \lim\limits_{j\in J}p_j(v)=v\text{ in }\Hc.
\end{equation} 
\item For all $j\in J$ one has $\pi(G)\Hc_j\subseteq\Hc_j$, which defines  
$\pi_j\colon G\to U(\Hc_j)$, $x\mapsto \pi(x)\vert_{\Hc_j}$. 
\end{enumerate}
Then $I_\pi=\overline{\bigcup\limits_{j\in J}I^0_{\pi_j}}=\overline{\bigcup\limits_{j\in J}I_{\pi_j}}$. 
If moreover $I_{\pi_j}$ is convex for all $j\in J$, then also $I_\pi$ is convex. 
\end{proposition}

\begin{proof}
For every $j\in J$, denoting by $\Hc_{j,\infty}$ the space of smooth vectors for the representation $\pi_j$, 
one has $\Hc_{j,\infty}=\Hc_j\cap\Hc_\infty$, and this implies $I^0_{\pi_j}\subseteq I^0_\pi$, hence 
$\overline{\bigcup\limits_{j\in J}I_{\pi_j}}\subseteq I_\pi$. 

To prove the converse inclusion, let $v\in\Hc_\infty$ arbitrary. 
For all $j\in J$ and $x\in G$ one has $p_j\pi(x)=\pi_j(x)p_j$, 
hence Remark~\ref{intert} implies $p_j(\Hc_\infty)\subseteq \Hc_{j,\infty}$. 
In particular $p_j(v)\in \Hc_{j,\infty}=\Hc_j\cap\Hc_\infty\subseteq\Hc_\infty$. 
We now prove that 
\begin{equation}\label{directed_proof_eq1}
\lim_{j\in J}p_j(v)=v\text{ in }\Hc_\infty.
\end{equation}
In fact, $(A(p_j(v)))(y)=\pi(y)p_j(v)=p_j\pi(y)v$ for all $y\in G$, 
and then \eqref{directed_eq1} implies \eqref{directed_proof_eq1}. 
Now, using Lemma~\ref{cont}, 
$$\Psi_\pi(v)=\lim_{j\in J}\Psi_\pi(p_j(v))=\lim_{j\in J}\Psi_{\pi_j}(p_j(v))\in 
\overline{\bigcup\limits_{j\in J}I^0_{\pi_j}}\subseteq\overline{\bigcup\limits_{j\in J}I_{\pi_j}}$$
and this completes the proof of the equalities from the statement. 

The convexity assertion follows from these equalities, 
since the closure of any upwards directed family of closed convex sets is in turn convex. 
\end{proof}

The following result was suggested by \cite[Lemma 2.1]{Wi92}.

\begin{proposition}\label{ortsum}
If $\pi_j\colon G\to U(\Hc_j)$ is any representation for $j=1,\dots,m$, 
and one defines $\pi:=\pi_1\oplus\cdots\oplus\pi_m$, 
then one has 
$$I^0_\pi=\{t_1f_1+\cdots+t_mf_m\mid  f_j\in I^0_{\pi_j},\ 0\le  t_j\le 1\text{ for }j=1,\dots,m, 
 \text{ and } t_1+\cdots+t_m=1\}.$$ 
If moreover $I_{\pi_j}$ is convex for $j=1,\dots,m$, then also $I_\pi$ is, and 
$I_\pi=\conv(I_{\pi_1}\cup \cdots\cup I_{\pi_m})$. 
\end{proposition}

\begin{proof}
Let $\Hc:=\Hc_1\oplus\cdots\oplus\Hc_m$, and $\Hc_{j,\infty}$ be the space of smooth vectors for the representation $\pi_j$ for $j=1,\dots,m$. 
Using Remark~\ref{intert}, one can see that 
$$\Hc_\infty=\{v_1\oplus \cdots\oplus v_m\mid v_j\in\Hc_{j,\infty}\text{ for }j=1,\dots,m\}.$$ 
Moreover, if $v_j\in\Hc_{j,\infty}\setminus\{0\}$ for $j=1,\dots,m$, 
and $v:=v_1\oplus \cdots\oplus v_m$, then it is easily checked that
\begin{equation}\label{ortsum_proof_eq1}
\Psi_\pi(v)
=\frac{\Vert v_1\Vert^2}{\Vert v\Vert^2}\Psi_{\pi_1}(v_1)
+\cdots+\frac{\Vert v_m\Vert^2}{\Vert v\Vert^2}\Psi_{\pi_m}(v_m)
\end{equation}
with $\Vert v\Vert^2=\Vert v_1\Vert^2+\cdots+\Vert v_m\Vert^2$, 
and the conclusion follows at once. 
\end{proof}

The following proposition is a generalization of \cite[Prop. 4.2]{AL92} 
from finite-dim\-ensio\-nal Lie groups to arbitrary topological groups, 
with a slightly different proof that relies on the above Proposition~\ref{directed}. 

\begin{proposition}\label{ortinfty}
If $\pi_j\colon G\to U(\Hc_j)$ is any representation for which  
$I_{\pi_j}$
is convex for all $j\in J$,  
then for the representation $\pi:=\bigoplus\limits_{j\in J}\pi_j$ 
one has 
$I_\pi=\overline{\conv(\bigcup\limits_{j\in J}I_{\pi_j})}$ and this is a convex set. 
\end{proposition}

\begin{proof}
Use Proposition~\ref{directed} for the family of finite orthogonal direct sums of representations $\pi_j$, 
and then compute the closed momentum set of each of these finite sums by means of Proposition~\ref{ortsum}. 
\end{proof}

The last results of the present section were suggested by \cite[Lemma 2.2]{Wi92} 
on finite-dim\-ension\-al representations. 

\begin{proposition}\label{pullback}
Let $\phi\colon G_1\to G_2$ be any morphism of topological groups 
and define $P\colon \RR^{\Lg(G_2)}\to\RR^{\Lg(G_1)}$, $P(f):=f\circ\Lg(\phi)$.  
For any representation $\pi\colon G_2\to U(\Hc)$ denote by $\Hc_\infty(\pi)$ and $\Hc_\infty(\pi\circ\phi)$ 
the spaces of smooth vectors of the representations $\pi$ and $\pi\circ\phi$, respectively. 
Then $\Hc_\infty(\pi)\subseteq \Hc_\infty(\pi\circ\phi)$ 
and $P(I^0_\pi)\subseteq I^0_{\pi\circ\phi}$. 
If moreover  $\Hc_\infty(\pi)$ is dense in $\Hc_\infty(\pi\circ\phi)$,  
then $\overline{P(I_\pi)}=I_{\pi\circ\phi}$. 
If in addition $I_\pi$ is convex , then also $I_{\pi\circ\phi}$ is. 
\end{proposition}

\begin{proof}
First note that 
for every $\psi\in\Ci(G_2,\Hc)$ one has $\psi\circ\phi\in\Ci(G_1,\Hc)$, 
and 
$$D^k(\psi\circ\phi)(x,\gamma_1,\dots,\gamma_k)=(D^k\psi)(\phi(x),\phi\circ\gamma_1,\dots,\phi\circ\gamma_k)$$
for all $x\in G_1$, $k\ge 1$ and $\gamma_1,\dots,\gamma_k\in\Lg(G_1)$ 
(see e.g., \cite[Th. 1.3.2.2(i')]{BCR81}). 

For every $v\in \Hc_\infty(\pi)$ one has $\pi(\cdot)v\in\Ci(G_2,\Hc)$, 
hence the above remark implies $(\pi\circ\phi)(\cdot)v\in\Ci(G_1,\Hc)$, 
that is, $v\in \Hc_\infty(\pi\circ\phi)$. 
Moreover, it is easily checked that 
\begin{equation}\label{pullback_proof_eq1}
\Psi_{\pi\circ\phi}(v)=\Psi_\pi(v)\circ\Lg(\phi)=P(\Psi_\pi(v)). 
\end{equation}
Now, if  $\Hc_\infty(\pi)$ is dense in $\Hc_\infty(\pi\circ\phi)$, 
then $I_{\pi\circ\phi}=\overline{\Psi_{\pi\circ\phi}(\Hc_\infty(\pi)\setminus\{0\})}$ by Lemma~\ref{cont}.  
Therefore, using \eqref{pullback_proof_eq1}, one obtains $I_{\pi\circ\phi}=\overline{P(I^0_\pi)}$, 
and the equality $\overline{P(I_\pi)}=I_{\pi\circ\phi}$ follows at once. 
Finally, if $I_\pi$ is convex, then its image through the linear map $P$ is convex, 
and so is the closure that image, that is, $I_{\pi\circ\phi}$. 
This completes the proof. 
\end{proof}

\begin{corollary}\label{finitedim}
Let $G$ be any topological group. 
If $\pi\colon G\to U(\Hc)$ is any representation with $\dim\Hc<\infty$, then $\Hc_\infty=\Hc$. 
\end{corollary}

\begin{proof} 
Regarding the identity map $\id_{U(\Hc)}\colon U(\Hc)\to U(\Hc)$ as a representation 
of the Lie group $U(\Hc)$ (using the hypothesis $\dim\Hc<\infty$), 
it is clear that $\Hc_\infty(\id_{U(\Hc)})=\Hc$. 
On the other hand, Proposition~\ref{pullback} implies $\Hc_\infty(\id_{U(\Hc)})\subseteq\Hc_\infty(\id_{U(\Hc)}\circ\pi)
=\Hc_\infty(\pi)$, hence $\Hc\subseteq\Hc_\infty(\pi)$, and we are done. 
\end{proof}

\section{Moment sets for representations of locally compact groups}

In this section we consider only representations of locally compact groups, 
since this hypothesis allows us to use $C^*$-algebras in the study of their representation theory. 

\begin{notation}
\normalfont
If $G$ is any locally compact group  
and $\pi\colon G\to U(\Hc)$ is any representation, 
then we denote again by $\pi$ both its corresponding representation of the Banach $*$-algebra $L^1(G)$ 
and its extension to the group $C^*$-algebra $C^*(G)$, 
that is, the enveloping $C^*$-algebra of $L^1(G)$. 
The kernel of the $*$-representation $\pi\colon C^*(G)\to\Bc(\Hc)$ is denoted by $\Ker_{C^*}(\pi)$. 
\end{notation}

The proofs of Propositions \ref{half}--\ref{other} below are merely adaptations of the method of proof of \cite[Prop. 5.1]{AL92} 
using also some ideas from \cite[Sect. X.6]{Ne00}.

\begin{proposition}\label{half}
If $G$ is any locally compact group and $\pi_j\colon G\to U(\Hc_j)$ for $j=1,2$ are any representations 
with $\Ker_{C^*}(\pi_1)\supseteq\Ker_{C^*}(\pi_2)$, 
then $I_{\pi_1}\subseteq\overline{\conv(I^0_{\pi_2})}\subseteq\overline{\conv(I_{\pi_2})}$. 
\end{proposition}

\begin{proof} 
The proof proceeds in three steps. 

Step 1. 
By Lemma~\ref{cont}, it suffices to check that for every vector $w$ in some dense subset of $\Hc_{1,\infty}$ 
one has 
\begin{equation}\label{half_proof_eq1}
\Psi_{\pi_1}(w)\in \overline{\conv(I^0_{\pi_2})}. 
\end{equation} 
To this end we will use the dense subset of $\Hc_{1,\infty}$ provided by Proposition~\ref{appr}, 
hence we will check \eqref{half_proof_eq1} for $w=\pi_1(f)v$, where $v\in\Hc_{1,\infty}$ with $\Vert v\Vert=1$, 
and $f\in\Cc_0^\infty(G)$. 

Step 2. At this step we fix $\gamma\in\Lg(G)$, 
and let $\varepsilon>0$ be also fixed for the moment.  
Since $\Ker_{C^*}(\pi_1)\supseteq\Ker_{C^*}(\pi_2)$, it follows by \cite[Prop. 3.4.2(i)]{Di64} 
that for every finite set $F\subseteq C^*(G)$ there exist an integer $n\ge 1$ and some vectors $\eta_1,\dots,\eta_n\in\Hc_2$ 
with $\sum\limits_{k=1}^n\Vert\eta_k\Vert^2=1$ and for all $a\in F$, 
\begin{equation}\label{half_proof_eq2}
\vert\langle\pi_1(a)v,v\rangle-\sum_{k=1}^n\langle\pi_2(a)\eta_k,\eta_k\rangle\vert<\varepsilon. 
\end{equation}
We use the above property for $F=\{f^*\ast f,f^*\ast D^R_\gamma f\}$, 
using the notation~\eqref{aux4_eq1_R}.  
For any $u\in\Hc_{j,\infty}$ and $a=f^*\ast D^R_\gamma f$ one has by \eqref{appr_eq1}, 
$$\begin{aligned}
\langle\pi_j(a)u,u\rangle
&=\langle\pi_j(f^*\ast D^R_\gamma f)u,u\rangle 
=
\langle\pi_j(D^R_\gamma f)u,\pi_j(f)u\rangle 
=\langle\de\pi_j(\gamma)\pi_j(f)u,\pi_j(f)u\rangle \\
&=\ie\Vert \pi_j(f)u \Vert^2(\Psi_{\pi_j}(\pi_j(f)u))(\gamma)
\end{aligned}$$ 
hence \eqref{half_proof_eq2} implies 
$$\Bigl\vert(\Psi_{\pi_1}(\pi_1(f)v))(\gamma)-\sum_{k=1}^m\frac{\Vert\pi_2(f)\eta_k\Vert^2}{\Vert\pi_1(f)v\Vert^2}(\Psi_{\pi_2}(\pi_2(f)\eta_k))(\gamma)\Bigr\vert<
\frac{\varepsilon}{\Vert\pi_1(f)v\Vert^2}. $$
Then 
\begin{equation}\label{half_proof_eq3}
\Bigl\vert(\Psi_{\pi_1}(\pi_1(f)v))(\gamma)
-\delta_\varepsilon\sum_{k=1}^m t_k(\Psi_{\pi_2}(\pi_2(f)\eta_k))(\gamma)\Bigr\vert<
\frac{\varepsilon}{\Vert\pi_1(f)v\Vert^2} 
\end{equation}	
where $t_k:=\Vert\pi_2(f)\eta_k\Vert^2/\sum\limits_{\ell=1}^n\Vert\pi_2(f)\eta_\ell\Vert^2$ 
and $\delta_\varepsilon:=\Bigl(\sum\limits_{\ell=1}^n\Vert\pi_2(f)\eta_\ell\Vert^2\Bigr)/\Vert\pi_1(f)v\Vert^2$, 
hence $t_1,\dots,t_n\in[0,1]$ and $t_1+\cdots+t_n=1$. 
In particular $\psi_\varepsilon:=\sum\limits_{k=1}^m t_k\Psi_{\pi_2}(\pi_2(f)\eta_k)\in\conv(I^0_{\pi_2})$. 

On the other hand, using \eqref{half_proof_eq2} for $a=f^*\ast f$, one obtains 
$$\Bigl\vert\Vert \pi_1(f)v\Vert^2-\sum\limits_{\ell=1}^n\Vert\pi_2(f)\eta_\ell\Vert^2\Bigr\vert<\varepsilon$$
that is, $\vert 1-\delta_\varepsilon\vert<\varepsilon/\Vert\pi_1(f)v\Vert^2$. 
Consequently, also taking into account
\eqref{half_proof_eq3}, we see that for $w=\pi_1(f)v$ and arbitrary $\gamma\in\Lg(G)$, 
there exist families $\{\delta_\varepsilon\}_{\varepsilon>0}$ in $(0,\infty)$ 
and $\{\psi_\varepsilon\}_{\varepsilon>0}$ in $\conv(I^0_{\pi_2})$ satisfying 
$\lim\limits_{\varepsilon\to0} \delta_\varepsilon=1$ and 
$\lim\limits_{\varepsilon\to0}\delta_\varepsilon\psi_\varepsilon(\gamma)=(\Psi_{\pi_1}(w))(\gamma)$, 
hence 
\begin{equation}\label{half_proof_eq4}
\lim\limits_{\varepsilon\to0}\psi_\varepsilon(\gamma)=(\Psi_{\pi_1}(w))(\gamma). 
\end{equation}
We will use this fact in the next step of the proof in order to obtain \eqref{half_proof_eq1}. 

Step 3. 
Assume that \eqref{half_proof_eq1} does not hold for $w=\pi_1(f)v$ as above. 
Since $\overline{\conv(I^0_{\pi_2})}$ is a closed convex subset 
of the locally convex space $\Lg(G)^*$ (see Lemmas \ref{lc-pro} and \ref{lincont}), 
it follows by \cite[Th. 3.4(b)]{Ru91} that there exist a real number $t>0$ 
a continuous linear functional $\Gamma\colon \Lg(G)^*\to \RR$ 
with 
$$\Gamma(\Psi_{\pi_1}(w)))+t<\Gamma(\psi)\text{ for all }\psi\in \overline{\conv(I^0_{\pi_2})}.$$ 
Since $\Lg(G)^*$ is a topological subspace of $\RR^{\Lg(G)}$, 
hence is endowed with the topology of pointwise convergence on $\Lg(G)$, 
it follows by \cite[Th. 3.10]{Ru91} that there exists $\gamma\in\Lg(G)$ 
such that $\Gamma(\psi)=\psi(\gamma)$ for all $\psi\in \Lg(G)^*$, 
and therefore 
$$(\Psi_{\pi_1}(w))(\gamma)+t<\psi(\gamma)\text{ for all }\psi\in \overline{\conv(I^0_{\pi_2})}.$$ 
If we use the above property for $\psi=\psi_\varepsilon$, we obtain a contradiction with \eqref{half_proof_eq4}, 
and we are done. 
\end{proof}

In connection with the proof of Proposition~\ref{half} we note that the use of Lemmas \ref{lc-pro} and \ref{lincont} 
could have been avoided by using \cite[Prop. 3.4.2(i)]{Di64} 
for a larger finite set $F=\{f^*\ast f\}\cup\{f^*\ast D^R_\gamma f\mid \gamma\in L_0\}$,  
with an arbitrary finite set $L_0\subseteq\Lg(G)$, 
since the continuous linear functionals on $\RR^{\Lg(G)}$ are the linear combinations of evaluations 
at arbitrary finte subsets of $\Lg(G)$. 

We now record a useful disintegration property of cyclic $*$-representations of separable $C^*$-algebras. 
This is essentially again a by-product of the proof of \cite[Prop. 5.1]{AL92}. 

\begin{lemma}\label{disintegr}
Let $\pi\colon\Ac\to\Bc(\Hc)$ be any cyclic $*$-representation of a separable $C^*$-algebra. 
Then there exist a compact metric space $Z$ 
with a probability Radon measure $\mu$, 
a measurable field of Hilbert spaces $\{\Hc_\zeta\}_{\zeta\in Z}$, 
a vector $\int^\oplus w_\zeta\de\mu(\zeta)\in \int^\oplus \Hc_\zeta\de\mu(\zeta)$, 
a measurable field of irreducible $*$-representations $\{\sigma_\zeta\colon \Ac\to U(\Hc_\zeta)\}_{\zeta\in Z}$ 
and a unitary operator $V\colon \Hc\to  \int^\oplus \Hc_\zeta\de\mu(\zeta)$ satisfying the following conditions: 
\begin{enumerate}
\item\label{disintegr_item1} 
One has  $V\pi(\cdot)V^{-1}=\int^\oplus\sigma_\zeta(\cdot)\de\mu(\zeta)$ 
and $w$ is a cyclic vector for that representation. 
\item\label{disintegr_item2} 
For every $a,b\in\Ac$ the function 
$\zeta\mapsto \langle \sigma_\zeta(a)w_\zeta,\sigma_\zeta(b)w_\zeta\rangle$ is continuous on $Z$. 
\end{enumerate}
\end{lemma}

\begin{proof}
Extending $\pi$ to the unitization of $\Ac$ if necessary, we may assume that $\Ac$ is unital. 
Let $w\in\Hc$ be any cyclic vector for $\pi$ with $\Vert v\Vert=1$, 
and define the state $\varphi\colon\Ac\to\CC$, $\varphi(\cdot):=\langle\pi(\cdot)v,v\rangle$. 
Denoting by $S$ the set of all states of $\Ac$, viewed as a compact metrizable space with the weak-topology 
(since $\Ac$ is separable), 
it follows by 
\cite[4.1.3, 4.1.11, 4.1.25, 4.2.5]{BR87} that there exists a probability Radon measure $\mu$ on $S$ which is an orthogonal measure and satisfies  
$$\varphi=\int_S\zeta\de\mu(\zeta) $$
and $\mu(P)=1$, where $P$ is the set of pure states of $\Ac$. 
Since $\mu$ is an orthogonal measure and $\pi$ is unitary equivalent to the GNS representation of $\Ac$ associated with the state $\varphi$ (since $v$ is a cyclic vector), it follows by \cite[4.4.9]{BR87} that 
there exists a unitary operator 
$V\colon \Hc\to  \int^\oplus \Hc_\zeta\de\mu(\zeta)$ 
with $V\pi(\cdot)V^{-1}=\int^\oplus\sigma_\zeta(\cdot)\de\mu(\zeta)$ 
and $Vw=\int^\oplus w_\zeta\de\mu(\zeta)$, 
where for every $\zeta\in S$ one denotes by $\sigma_\zeta\colon\Ac\to\Bc(\Hc_\zeta)$ its corresponding GNS representation with the unit cyclic vector $w_\zeta\in\Hc_\zeta$ with $\zeta(\cdot)=\langle\sigma_\zeta(\cdot)w_\zeta,w_\zeta\rangle$. 
Then for every $a,b\in\Ac$ and $\zeta\in S$ one has 
$$\langle \sigma_\zeta(a)w_\zeta,\sigma_\zeta(b)w_\zeta\rangle=\langle\sigma_\zeta(b^*a)w_\zeta,w_\zeta\rangle=\zeta(b^*a)$$
and this depends continuously on $\zeta\in S$ since $S$ is endowed with the weak topology. 
Then the conclusion is obtained if we define $Z$ as the closure of $P$ in $S$.  
\end{proof}

\begin{proposition}\label{other}
Let $G$ be any separable locally compact group and $\pi\colon G\to U(\Hc)$ be any representation 
such that $I_{\sigma}$ is convex for every irreducible representation 
$\sigma\colon G\to U(\Hc_\sigma)$ with $\Ker_{C^*}(\sigma)\supseteq\Ker_{C^*}(\pi)$. 
Then $I_\pi$ is convex. 
\end{proposition}

\begin{proof}
Let $\pi=\bigoplus_{j\in J}\pi_j$ be a decomposition of $\pi\colon C^*(G)\to\Bc(\Hc)$ 
into cyclic representations. 
By Proposition~\ref{ortinfty}, it suffices to prove that for every $j_0\in J$ the set $I_{\pi_{j_0}}$ is convex. 
One has $\Ker\pi=\bigcap\limits_{j\in J}\Ker\pi_j\subseteq\Ker\pi_{j_0}$,  
hence every $\sigma\colon G\to U(\Hc_\sigma)$ with $\Ker_{C^*}(\sigma)\supseteq\Ker_{C^*}(\pi_{j_0})$  
also satisfies $\Ker_{C^*}(\sigma)\supseteq\Ker_{C^*}(\pi)$, hence $I_{\sigma}$ is convex by hypothesis. 
Thus $\pi_{j_0}$ satisfies the condition from the hypothesis and then, 
replacing $\pi$ by $\pi_{j_0}$, we may assume that $\pi$ is a cyclic representation. 

But then, using Lemma~\ref{disintegr}, 
we may also assume that there exist a compact metric space $Z$ 
with a positive measure $\mu$, 
a measurable field of Hilbert spaces $\{\Hc_\zeta\}_{\zeta\in Z}$ and a measurable field of irreducible representations $\{\sigma_\zeta\colon G\to U(\Hc_\zeta)\}_{\zeta\in Z}$ 
such that $\Hc=\int^\oplus\Hc_\zeta\de\mu(\zeta)$ and $\pi=\int^\oplus\sigma_\zeta\de\mu(\zeta)$. 

Using the method of proof of \cite[Th. X.6.16(ii)]{Ne00}, 
one can find a $\mu$-null set $Z_0\subseteq Z$ with $\Ker_{C^*}(\pi)=\Ker_{C^*}(\widetilde{\pi})$, 
where $\widetilde{\pi}:=\bigoplus\limits_{\zeta\in Z\setminus Z_0}\sigma_\zeta$. 
Then the representation $\widetilde{\pi}$ is convex by Proposition~\ref{ortinfty},  
and by Proposition~\ref{half} one further obtains $I_\pi\subseteq I_{\widetilde{\pi}}$. 

To prove that $I_{\pi}$ is convex, we will now check the converse inclusion $I_\pi\supseteq I_{\widetilde{\pi}}$. 
Using Proposition~\ref{ortinfty} and its proof, it suffices to prove that for every finite set 
$F=\{\zeta_1,\dots,\zeta_m\}\subseteq Z\setminus Z_0$ 
the representation $\widetilde{\pi}_F:=\bigoplus\limits_{j=1}^m\sigma_{\zeta_j}$ 
satisfies $I_{\widetilde{\pi}_F}\subseteq I_{\pi}$. 
To this end, using Proposition~\ref{appr} and Lemma~\ref{cont}, 
it suffices to prove that for arbitrary 
$v=v_{\zeta_1}\oplus\cdots\oplus v_{\zeta_m}\in \bigoplus\limits_{j=1}^m\Hc_{\zeta_j}$ 
and $f\in \Cc_0^\infty(G)$ with $\Vert \widetilde{\pi}_F(f)v\Vert=1$ 
one has 
\begin{equation}\label{other_proof_eq-1}
\Psi_{\widetilde{\pi}_F}(\widetilde{\pi}_F(f)v)\in I_{\pi}. 
\end{equation}
Using \eqref{ortsum_proof_eq1} and $\Vert \widetilde{\pi}_F(f)v\Vert=1$, one obtains 
\begin{align}
\Psi_{\widetilde{\pi}_F}(\widetilde{\pi}_F(f)v)
&=\Psi_{\widetilde{\pi}_F}(\sigma_{\zeta_1}(f)v_{\zeta_1}\oplus\cdots\oplus \sigma_{\zeta_m}(f)v_{\zeta_m}) \nonumber \\
&\label{other_proof_eq0}
=\sum_{j=1}^m\Vert \sigma_{\zeta_j}(f)v_{\zeta_j}\Vert^2\Psi_{\sigma_{\zeta_j}}(\sigma_{\zeta_j}(f)v_{\zeta_j}) 
\end{align}
For every $\varepsilon>0$ let $h^\varepsilon:=(h^\varepsilon_1,\dots,h^\varepsilon_m)\in \Cc(Z)^m$ with 
$0\le h^\varepsilon_j$ on $Z$, $h^\varepsilon_j=1$ on some neighborhood of $\zeta_j\in Z$, 
$\supp h^\varepsilon_j$ contained in the open ball $B_{\zeta_j}(\varepsilon)$ in $Z$ 
of center $\zeta_j$ and radius~$\varepsilon$, 
$\Vert h^\varepsilon_j\Vert_{L^2(Z,\mu)}=1$, 
and  
$(\supp h^\varepsilon_j)\cap(\sup h^\varepsilon_k)=\emptyset$
if $j\ne k$. 
Since the representation $\sigma_{\zeta_j}\colon C^*(G)\to\Bc(\Hc_{\zeta_j})$ is irreducible, 
it follows by Kadison's transitivity theorem \cite[Cor. 2.8.4]{Di64} 
that there exists $a_j\in C^*(G)$ with $\sigma_{\zeta_j}(a_j)w_{\zeta_j}=v_{\zeta_j}$, 
where $\int^\oplus w_\zeta\de\mu(\zeta)\in\Hc$ is the unit vector 
given by Lemma~\ref{disintegr}.  
Define 
$$v^\varepsilon:=\sum_{j=1}^m\int^\oplus_Z h^\varepsilon_j(\zeta)\sigma_\zeta(a_j)w_\zeta\de\mu(\zeta)\in\Hc,$$
hence 
\begin{equation}\label{other_proof_eq1}
\pi(f)v^\varepsilon=\sum_{j=1}^m\int^\oplus_Z h^\varepsilon_j(\zeta)\sigma_\zeta(f)\sigma_\zeta(a_j)w_\zeta\de\mu(\zeta). 
\end{equation}
Since $(\supp h^\varepsilon_j)\cap(\sup h^\varepsilon_k)=\emptyset$
for $j\ne k$, it follows that 
$$\begin{aligned}
\Vert \pi(f)v_\varepsilon\Vert^2
&=\sum_{j=1}^m\int_Z h^\varepsilon_j(\zeta)^2\Vert \sigma_{\zeta}(f)\sigma_\zeta(a_j)w_\zeta\Vert^2\de\mu(\zeta) \\
&=\sum_{j=1}^m\Vert \sigma_{\zeta_j}(f)v_{\zeta_j}\Vert^2
+\sum_{j=1}^m\int_Z h^\varepsilon_j(\zeta)^2
(\Vert \sigma_{\zeta}(fa_j)w_\zeta\Vert^2-\Vert \sigma_{\zeta_j}(fa_j)w_{\zeta_j}\Vert^2)\de\mu(\zeta)
\end{aligned}$$
since $\sigma_{\zeta_j}(fa_j)w_{\zeta_j}=\sigma_{\zeta_j}(f)\sigma_{\zeta_j}(a_j)w_{\zeta_j}=\sigma_{\zeta_j}(f)v_{\zeta_j}$ and $\Vert h^\varepsilon_j\Vert_{L^2(Z,\mu)}=1$. 
Now recalling that $\supp h^\varepsilon_j\subseteq B_{\zeta_j}(\varepsilon)$ 
and the functions $\zeta\mapsto\Vert \sigma_\zeta(fa_j)w_\zeta\Vert$ are continuous for $j=1,\dots,m$ 
(see Lemma~\ref{disintegr}\eqref{disintegr_item2}), 
and using Lebesgue's dominated convergence theorem, 
it follows that 
\begin{equation}\label{other_proof_eq2}
\lim_{\varepsilon\to 0}\Vert \pi(f)v_\varepsilon\Vert^2=\sum_{j=1}^m\Vert \sigma_{\zeta_j}(f)v_{\zeta_j}\Vert^2
=\Vert \widetilde{\pi}_F(f)v\Vert^2=1.
\end{equation}
Moreover, using~\eqref{appr_eq1}, one has for arbitrary $\gamma\in\Lg(G)$ 
$$\begin{aligned}
\langle\de\pi(\gamma)\pi(f)v_\varepsilon,\pi(f)v_\varepsilon\rangle
&=\langle\pi(D^R_\gamma f)v_\varepsilon,\pi(f)v_\varepsilon\rangle \\
&=\sum_{j=1}^m\int_Z h^\varepsilon_j(\zeta)^2\langle \sigma_{\zeta}(D^R_\gamma f)\sigma_\zeta(a_j)w_\zeta, \sigma_{\zeta}(f)\sigma_\zeta(a_j)w_\zeta\rangle\de\mu(\zeta) 
\end{aligned}$$
where we also used \eqref{other_proof_eq1} and the fact that $(\supp h^\varepsilon_j)\cap(\sup h^\varepsilon_k)=\emptyset$
for $j\ne k$. 
Since the functions 
$$\zeta\mapsto \langle \sigma_{\zeta}(D^R_\gamma f)\sigma_\zeta(a_j)w_\zeta, \sigma_{\zeta}(f)\sigma_\zeta(a_j)w_\zeta\rangle
=\langle \sigma_{\zeta}((D^R_\gamma f)a_j)w_\zeta, \sigma_{\zeta}(fa_j)w_\zeta\rangle$$ 
are continuous by Lemma~\ref{disintegr}\eqref{disintegr_item2}, one obtains as above by \eqref{appr_eq1}
$$\begin{aligned}
\lim_{\varepsilon\to 0}\langle\de\pi(\gamma)\pi(f)v_\varepsilon,\pi(f)v_\varepsilon\rangle
&=\sum_{j=1}^m \langle \sigma_{\zeta_j}(D^R_\gamma f)\sigma_{\zeta_j}(a_j)w_{\zeta_j}, 
\sigma_{\zeta_j}(f)\sigma_{\zeta_j}(a_j)w_{\zeta_j}\rangle \\
&=\sum_{j=1}^m \langle \de\sigma_{\zeta_j}(\gamma)\sigma_{\zeta_j}(f)v_j, 
\sigma_{\zeta_j}(f)v_j\rangle \\
&=(\Psi_{\widetilde{\pi}_F}(\widetilde{\pi}_F(f)v))(\gamma)
\end{aligned}$$ 
where the latter equality follows by \eqref{other_proof_eq0}. 
Now, by \eqref{other_proof_eq2}, one has 
$$\lim_{\varepsilon\to 0}(\Psi_\pi(\pi(f)v_\varepsilon))(\gamma)
=(\Psi_{\widetilde{\pi}_F}(\widetilde{\pi}_F(f)v))(\gamma).$$
Since $\gamma\in\Lg(G)$ is arbitrary, this implies \eqref{other_proof_eq-1}  and completes the proof. 
\end{proof}

\begin{corollary}\label{red}
Let $G$ be any separable locally compact group. 
If the closed moment set of every irreducible representation of $G$ is convex, 
then this property is shared by every representation of $G$.  
\end{corollary}

\begin{proof}
Use Proposition~\ref{other}. 
\end{proof}

\section{Main result}

It is convenient to make the following definition. 
More detailed information on solvable topological groups can be found in \cite[Chs. 7 and 10]{HM07} and 
\cite{Bos76}. 

\begin{definition}
\normalfont
A topological group $G$ is called \emph{solvable} 
if for every neighborhood $V$ of $\1\in G$ there exists an integer $k\ge 0$ with $G^{(k)}\subseteq V$. 
Here $G=G^{(0)}\supseteq G^{(1)}\supseteq\cdots$ is the descending derived series of $G$ 
defined by the condition that $G^{(k+1)}$ is the closed subgroup of $G$ 
generated by the set $\{xyx^{-1}y^{-1}\mid x,y\in G^{(k)}\}$. 
\end{definition}

\begin{theorem}\label{main}
Let $G$ be any solvable separable locally compact group.  
Then the closed moment set of every representation of $G$ is convex. 
\end{theorem}

\begin{proof} 
We may assume that $G$ is connected, since $\Lg(G)$ depends only on the connected $\1$-component of $G$.
As Corollary~\ref{red} shows, it suffices to prove that for every
irreducible representation $\pi\colon G\to U(\Hc)$ its closed moment set $I_\pi$ is convex. 
Using  \cite[Cor. to Lemma 1]{Mag81}, 
there exists $N_0\in\Nc_0(G)$ with $\pi=\pi_0\circ p$, 
where we use notation from Corollary~\ref{pro}. 
By that corollary, $\Hc_\infty(\pi_0)$ is dense in $\Hc_\infty(\pi)$, 
hence Proposition~\ref{pullback} shows that 
if we check that $I_{\pi_0}$ is convex, then $I_\pi$ is convex as well.  
But $I_{\pi_0}$ is a convex set by \cite[Th. 13]{AL92},  
since $\pi_0\colon G/N_0\to U(\Hc)$ is a representation of 
a connected finite-dimensional Lie group which is solvable by \cite[Th. 10.18]{HM07}.  
This completes the proof. 
\end{proof}

\subsection*{Acknowledgment}
This research was partially supported from the Grant of the Romanian National Authority for Scientific Research, CNCS-UEFISCDI, project number PN-II-ID-PCE-2011-3-0131.

\end{document}